\documentclass[english]{article}
\usepackage{amsthm} 
\usepackage{amsmath}
\usepackage{amssymb} 
\usepackage{quiver}
\usepackage{xspace}
\usepackage{stmaryrd}
\usepackage{thmtools} 
\usepackage{mathtools} 
\usepackage{float}
\usepackage{caption}
\usepackage{subcaption}

\newtheorem{definition}{Definition}
\newtheorem{theorem}{Theorem}

\newtheorem{lemma}{Lemma}
\newtheorem{proposition}{Proposition}

\newcount\foreachcount

\makeatletter
\let\ea\expandafter
\def\foreachLetter#1#2#3{\foreachcount=#1
  \ea\loop\ea\ea\ea#3\@Alph\foreachcount
  \advance\foreachcount by 1
  \ifnum\foreachcount<#2\repeat}
\def\definecal#1{\ea\gdef\csname c#1\endcsname{\ensuremath{\mathcal{#1}}\xspace}}
\foreachLetter{1}{27}{\definecal}

\makeatletter
\let\ea\expandafter
\def\foreachLetter#1#2#3{\foreachcount=#1
  \ea\loop\ea\ea\ea#3\@Alph\foreachcount
  \advance\foreachcount by 1
  \ifnum\foreachcount<#2\repeat}
\def\definecal#1{\ea\gdef\csname b#1\endcsname{\ensuremath{\mathbf{#1}}\xspace}}
\foreachLetter{1}{27}{\definecal}

\newcommand{\R}{\mathbb{R}}

\newcommand{\Vect}{\mathbf{Vect}_{\mathbb{R}}}
\newcommand{\Smooth}{\mathbf{Smooth}}

\newcommand{\Mat}{\mathbf{Mat}}
\newcommand{\Span}{\mathbf{Span}}
\newcommand{\FinSet}{\mathbf{FinSet}}

\newcommand{\CoKl}{\mathsf{CoKl}}
\newcommand{\Para}{\mathbf{Para}}
\newcommand{\Lens}{\mathbf{Lens}}

\newcommand{\Cat}{\mathbf{Cat}}

\newcommand{\comp}{\fatsemi}

\newcommand{\internalImpll}[3]{[#2, #3]_{#1}}

\newcommand{\internal}[3]{\internalImpll{#1}{#2}{#3}}

\newcommand{\adjobj}{\R^{n \times n}}
\newcommand{\plc}{\Para_\odot(\Lens_A(\CoKl(\adjobj \times -)))}
\newcommand{\pc}{\Para_\odot(\CoKl(\adjobj \times -))}
\newcommand{\src}{\textrm{src}}
\newcommand{\tgt}{\textrm{tgt}}

\newcommand{\GCNN}{\mathbf{GCNN}}
\newcommand{\Adj}{\mathrm{Adj}}

\newcommand{\diag}{\mathrm{diag}}
\newcommand{\ReLU}{\mathrm{ReLU}}

\title{Graph Convolutional Neural Networks as Parametric CoKleisli morphisms}
\author{Bruno Gavranovi\'c and Mattia Villani}

\begin{document}
\maketitle
\begin{abstract}
  We define the bicategory of Graph Convolutional Neural Networks $\GCNN_n$ for
  an arbitrary graph with $n$ nodes.
  We show it can be factored through the already existing categorical
  constructions for deep learning called $\Para$ and $\Lens$ with the base
  category set to the CoKleisli category of the product comonad.
  We prove that there exists an injective-on-objects, faithful 2-functor $\GCNN_n \to \pc$.
  We show that this construction allows us to treat the adjacency matrix of a
  GCNN as a global parameter instead of a a local, layer-wise one.
  This gives us a high-level categorical characterisation of a particular kind of inductive bias GCNNs possess.
  Lastly, we hypothesize about possible generalisations of GCNNs to general message-passing graph neural networks, connections to equivariant learning, and the (lack of) functoriality of activation functions.
\end{abstract}

	\section{Introduction}

Neural networks have recently been formalised in category theory using the abstraction of
parametric lenses \cite{GradientBasedLearning}.
This is a formalism that uses the categorical constructions of $\Para$ and $\Lens$ to compositionally model neural network weights and the backpropagation process, respectively.
By composing $\Para$ and $\Lens$ together for the setting of a category with sufficient structure the authors how general neural networks can be modelled.

This formalism is sufficiently powerful to encompass several families of neural network architectures: feedforward networks, recurrent networks, convolutional neural networks, residual networks, graph neural networks, generative adversarial networks, and many more.
These families of architectures differ substantially in character: they require as input different types of underlying datasets, they have different expressive power, and the inductive biases that characterise them vary vastly.
Numerous architectures have given birth to their sub-fields in each of whose results are being compared against different benchmark datasets and benchmark architectures. 

However, there is currently no way to see these high-level differences in
category theory.
In the current categorical formalism of \cite{GradientBasedLearning}, these different kinds of networks are still just parametric lenses.
Indeed, the abstraction lacks the sufficient resolution to distinguish architectures, while these yield remarkably different results when applied to their respective learning problems (recurrent nets for time series, convolutional nets in image recognition, etc.).
This creates difficulties on two fronts. When it comes to implementation in code, this framework only answers questions about composing networks
in the abstract form, and does not make it easier to \emph{implement} particular
families of architectures.
For instance, when it comes to Graph Convolutional Neural Networks which involve
the adjacency matrix of a particular graph, it's up to the programmer to
specify at which layers the graph convolution must be performed.
In the context of recurrent neural networks, the user has to make sure to
correctly unroll the network, and perform stateful computation in time.
Of course, many frameworks provide out-of-the-box solutions, but there is no
formal way to verify whether they satisfy the constraints of a particular
architecture. Often there are numerous variations to architecture one can
perform, and no principled approach to make high-level architectural design decisions.
In other words, it is not clear how to formally define the \emph{type of a
  neural network}: the different types of architectures have yet to be formalised in terms of type theory or category theory.

\textbf{Contributions.} In this paper, we make first steps in expressing these architectural
differences in terms of category theory. We explicitly focus on Graph
Convolutional Neural Networks and, given a graph $G$, we formally define a
bicategory of graph convolutional neural networks $\GCNN_G$.
We show this bicategory can be seen as arising out of a composition of smaller
semantic components previously introduced in categorical deep learning
literature: the framework of $\Para$ and $\Lens$.
We show that one can correctly embed $\GCNN_n$ in the aforementioned framework, when instantiated on the base category $\CoKl(A \times -)$.
This gives a birds-eye view of a particular kind of inductive bias GCNNs posses
- that of a globally available parameter - the adjacency matrix of the
underlying graph they're trained on.

\textbf{Acknolwedgements.} We thank Matteo Capucci, Petar Veli{\v c}kovi\'c and Andrew Dudzik for inspiring conversations.


	\section{Why Graphs?}

We proceed to describe how classical feedforward neural networks work, and then
the intuition behind generalising them to graph-based ones.

\subsection{Neural Networks}
A neural network layer in its simplest form involves a function of type
\begin{equation}
  \label{eq:nn_layer}
  f : P \times X \to Y
\end{equation}

This function takes in a \emph{parameter} (for instance, a matrix $W :
\R^{n \times m}$ of \emph{weights}), an input (for instance, a vector $x : \R^n$ of
\emph{features}), and computes an output value, usually passed on to the next
layer.
A layer is usually implemented as a function $(W, x) \mapsto \sigma(x^T W)$ which
first performs the linear operation of matrix multiplication and then applies a
non-linear activation function $\sigma : \R^m \to \R^m$.\footnote{This
  activation function is often actually a \emph{parallel product} of a
  activation functions. That is, we usually think of it as $\sigma^n : \R^n \to
  \R^n$, where $\sigma : \R \to \R$ is a non-linear activation such as sigmoid,
  or ReLU. The only exception known to the author to this case is the softmax
  function, defined in \cite{GradientBasedLearning}[Example 3.6]. Nevertheless, we can decompose this function with a parallel step and an aggregation step.}

Even though $f$ looks like a standard two-argument function, it's important to
note that the parameter $P$ (the weights) is an extra, ``hidden'' input that
belongs to the local data of the neural network itself, not the outside ports. This
means that each layer of a neural network can potentially have a different
parameter (weight) space.

If we are in possession of two layers of neural network, i.e. functions $f : P
\times X \to Y$ and $g : Q \times Y \to Z$, then their composition is defined to
be the following parameterised map
\begin{align*}
  (Q \times P) \times X &\to Y\\
  ((q, p), x) &\mapsto g(q, f(p, x))
\end{align*}

As noted in \cite{GradientBasedLearning}, this is amenable to formulation with
the $\Para$ construction, defined for any base monoidal category.
We refer the interested reader to the aforementioned paper for full details on
its construction, and here we only remark on the fact that $\Para$ comes
equipped with a graphical language, allowing us to talk about horizontal
composition of morphisms (composition of layers of a neural network) and
vertical reparameterisation of morphisms (reparameterisation of neural network
weights).
Figure \ref{fig:architecture_agnostic} shows an example of this graphical
language in action, depicting a composition of three morphisms in
$\Para(\cC)$. The composite map is parameterised by $P_1 \times P_2 \times
P_3$, the product of parameter spaces of its constituents.

\begin{figure}[H]
  \centering
  \includegraphics[width=.7\textwidth]{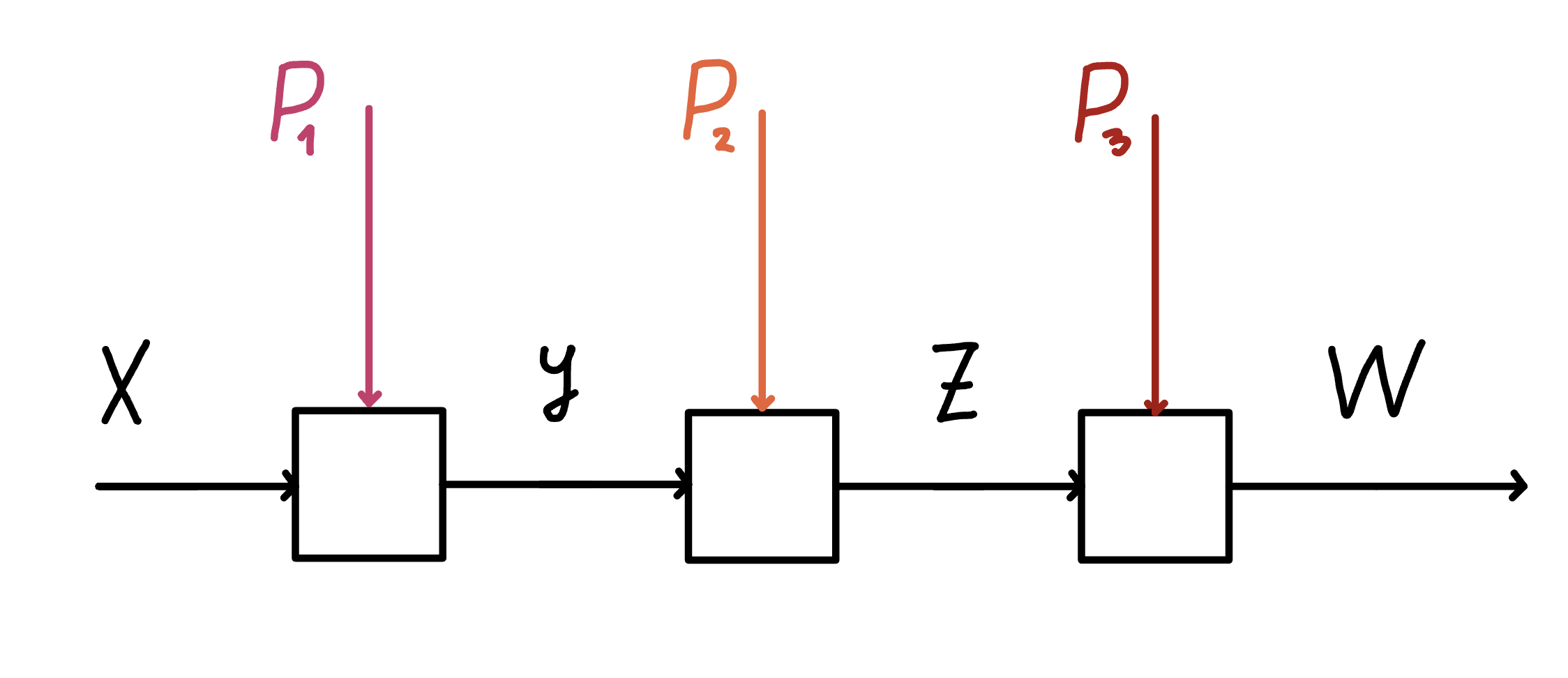}
  \caption{Graphical representation of a composition of three morphisms in $\Para(\cC)$. }
  \label{fig:architecture_agnostic}
\end{figure}

\subsection{Graph Neural Networks: Adding an Extra Dimension}

Neural networks described above model functions which process \emph{one} vector of features in a sequence of layers.
For instance, in the context of application of deep learning to house price
prediction, this input vector $x : \R^n$ could represent $n$ aspects of a real
estate property such as land area, number of bedrooms, years since last
renovation, and so on. The output value could be $\R$, which we can interpret as
an estimate of a price, or likelihood of sale.

There are two ways to think about generalising this story to a graph-based one.

\begin{itemize} 
\item \textbf{Location.} The value of a house is not
  fully determined by its intrinsic factors (such
  as number of bedrooms or its EPC rating); it also
  depends on extrinsic factors, such as \emph{features of other houses}
  around it. An expensive house in a neighbourhood of broken-down houses is
  worth less than the same house in an up-and-coming neighbourhood.
  More broadly, we expect houses nearby to be affect by similar exogenous factors;
  for instance, market dynamics forcing a no-arbitrage condition, which in this case
  resembles a continuity condition on location.

  Our feedforward network was fully determining the price of a property only by
  its intrinsic factors. In other words, we index by the trivial graph
  with one node and one edge, and there are no neighbourhoods or other
  next-door houses to reference. But we can index by arbitrary graphs
  (that, for instance, encode property adjacency information) and start
  determining attributes by looking at both a node and its neighbourhood.
  This means that we will be processing many input vectors, and use the
  connectivity of the graph to pass messages from nodes to their neighbours.

\item \textbf{Connecting the dots.} Another viewpoint that does not involve a graph,
  but still sheds light on the architectural distinction between a graph neural network
  and a vanilla feed-forward neural network, comes from looking at how batch computation
  is performed in the deep learning literature.
  To exploit the fast parallel nature of GPUs, input vectors are often
  processed many at a time, in parallel. This is called \emph{batching} and it
  adds an extra mode or dimension to the input of our neural network system. For instance, if the
  number of vectors in a batch is $b$, then our input gains another dimension
  and becomes a \emph{matrix} of type $X : \R^{b \times n}$, interpreted as a
  stacking $n$ vectors. To process this matrix, we add another dimension to our
  neural network: we make $b$-number of copies of the same network and process
  each $n$-sized vectors \emph{in parallel}, with the same parameter. The
  results in $b$ estimates of prices of these properties, encoded as a vector of
  type $\R^b$. This corresponds to indxexing by a \emph{discrete graph} with $n$
  nodes, where each node is one training example.
  But processing datapoints independently of each other in a batch ignores any relationships that might exist \emph{between} datapoints. In many cases, there's
  tangible relationships between datapoints, and the entire procedure can be
  generalised beyond discrete graphs.
\end{itemize}

In summary, we may view the above examples represent two conceptualisations of information flow through the architecture of a network, the former through the encoding of neighbourhoods of a graph by an adjacency matrix, which allows us to have parallel but interacting networks, fit for graph based problems; the latter formalises batch computation through the artefact of a graph with no edges.
Both of these stories partially overlap, and are unified in their requirements that our datasets have to now be living \emph{over graphs}. 
This is often the case, and as one warms up to this idea, they can easily start noticing that many common examples of architectures and datasets that appear to
have nothing to do in graphs are in fact inherently tied to them.

The way to think about information flow in a GNN is that information is
processed in layers, just like before.
But unlike before, where input and output ports were $1$-indexed euclidean
spaces, here input and output ports are going to be $V$-indexed euclidean spaces
and fibrewise smooth maps.
That is, we will have an extra layer of indexing throughout our computation.
This index corresponds to a node, and a fiber over it corresponds to the feature
space of that node.
In a single layer, each such broadcasts feature information to all its
neighbours.
At the receiving end, each node receives feature information from edges on its
neighbours, after which it \emph{aggregates} that information, and then uses it
compute an updated feature for itself (of possibly different dimensionality), as input to the next layer.

What is described here is in machine learning literature known as a \emph{Graph
  Convolution Neural Network} (GCNN), a special case of more general
\emph{message-passing} graph neural networks \cite{MPForQuantumChemistry, GNNDynProg}.
While a lot can be said about the intricacies of general message pasing neural
networks, in this paper we focus on GCNNs they can be written down in a rather
simple mathmatical form. Below we write the formula corresponding to a
\emph{single GCNN layer}, using the representation of an $n$-node graph through
its adjacency matrix $A$, though of as an object of $\R^{n \times n}$.

\begin{definition}[GCNN Layer]\label{def:prim_gcnn_layer}
Given input and output types  $\R^{n \times k}$ and $ \R^{n
  \times k'}$, a single GCNN layer between them is uniquely defined
as a smooth function
  \begin{align*}
    f : \R^{k \times k'} \times \R^{n \times n} \times \R^{n \times k} &\to \R^{n \times k'}\\
    (W, A, X) &\mapsto \sigma(AXW)
  \end{align*}
\end{definition}

The relationship between the left multiplication by the adjecency matrix and the
aforementioned broadcasting mechanism present in GCN as mentioned above
is demystified in \cite{kipf2016semi}, to which we refer the interested reader.
In summary, here we remark that $A$ plays the role of selecting
which nodes communicate with one another in the network,
represented by a morphisms of type $\R^{n \times k} \to \R^{n \times k'}$.
Indeed, the other terms in the domain of the GCN layer morphism
represent the parameter matrix of type $\R^{k\times k'}$
and the adjacency matrix of type $\R^{n \times n}$respectively


Two primitive layers can be composed, yielding a two-layer graph convolutional neural
network.\footnote{An underlying assumption in what follows is that the two layers have the same graph as an inductive bias. While restrictive as an abstraction, we can ascertain that in practice it is rare to want to change the graph topology in between layer computations, with the notable exception of \textit{graph pooling}, whereby the graph is coarsened to learn global behaviours on the graph representations as well as reducing memory and computational burden (see \cite{wang2020haar} for a well-known example).}


\begin{definition}[Composition of primitive GCNN
  layers]\label{def:prim_composition}
  Given two GCNN layers $\R^{n \times k} \to \R^{n \times k'}$ and $\R^{n \times
  k'} \to \R^{n \times k''}$ with implementation functions $f$ and $g$, respectively,
  we define a \emph{two-layer} Graph Convolutional Neural Network of type $\R^{n \times
    k} \to \R^{n \times k''}$ whose implementation is a smooth function
  \begin{align*}
    h : (\R^{k \times k'} \times \R^{k' \times k''}) \times \R^{n \times n} \times \R^{n \times k} &\to \R^{n \times k''}\\
    ((W, W'), A, X) &\mapsto g(W', A, f(W, A, X))
  \end{align*}

\end{definition}

\begin{figure}[H]
  \centering
  \includegraphics[width=\textwidth]{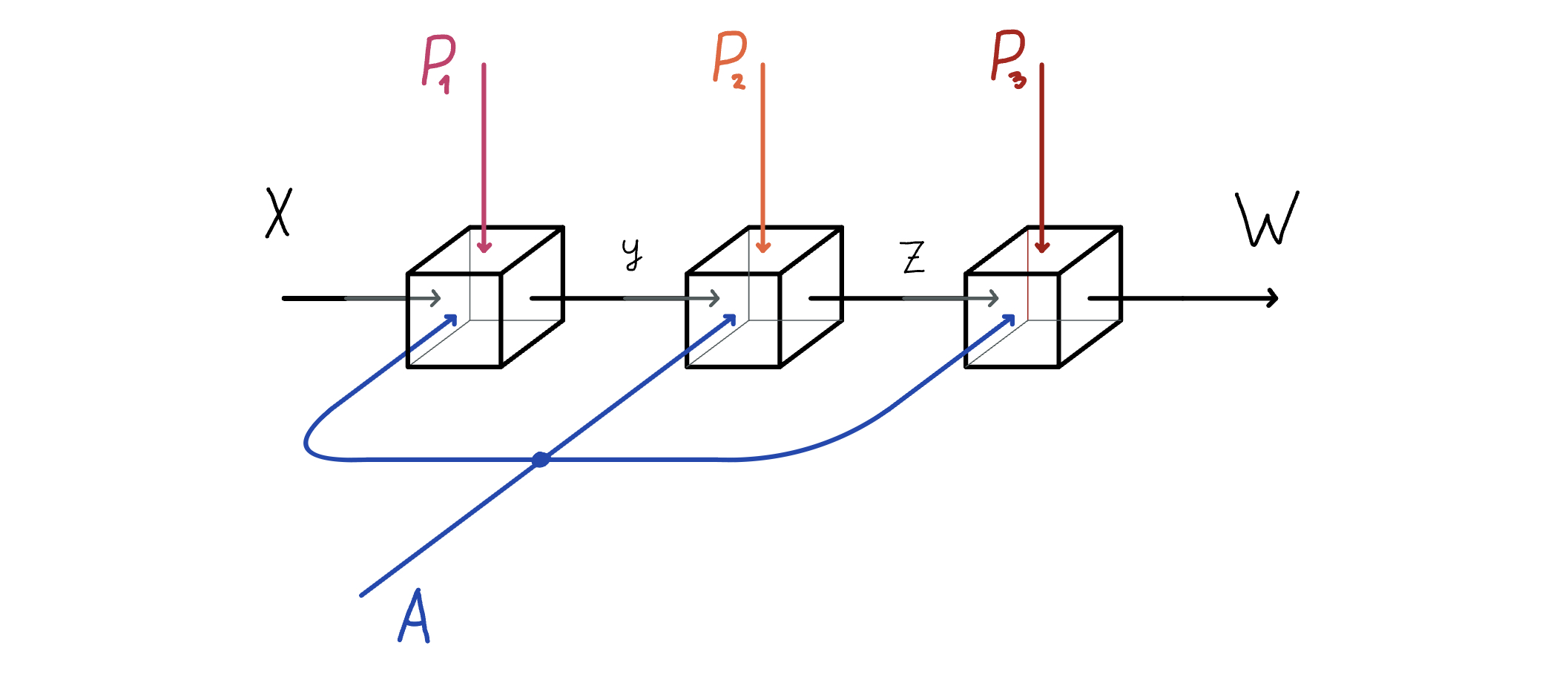}
  \caption{Composition of three layers in a graph convolutional neural network.
    On the vertical axis, we see three parameter spaces: $P_1$, $P_2$, and $P_3$
    of each layer. On the horizontal axis, we see the adjacency matrix $A$ as a
    ``parameter'' to each layer. By composing  new layers on the right the number
    of parameters increases, but the adjacency matrix is just copied.}
  \label{fig:graph_convolutional}
\end{figure}

The definition above can be extended to a Graph Convolutional Neural Network with an arbitrary number of layers (Figure \ref{fig:graph_convolutional}).
This finally allows us to define a general category of
$n$-node graph neural networks.
This construction is actually a \emph{bicategory}, as composition isn't strictly
associative and unital, and is more naturally thought of with 2-cells.
For presentation purposes we include the definition of 2-cells only prior to
stating Theorem \ref{thm:two_functor}.

\begin{definition}[Bicategory of Graph Convolutional Neural Networks]
We define $\GCNN_n$, the bicategory of Graph Neural Networks on graphs with $n$
nodes. Its objects are euclidean spaces of the form $\R^{n \times k}$
for some $k$, and a morphism $\R^{n \times k} \to \R^{n \times k'}$ is a graph
convolutional neural \emph{network}, i.e. a sequence of $m$ primitive GCNN layers.
We define 2-cells in Def. \ref{def:two_cells}.
\end{definition}

It is straightforward, albeit tedious to show that coherence conditions are
satisfied. While this might be the first time a graph convolutional neural
networks have been defined as a bicategory, this definition is merely a starting
point for this paper.
It has a number of intricate properties, and is defined within a very concrete setting.
In what follows, we proceed to ask the question: is there a
factorisation of the bicategory $\GCNN_n$ into sensible semantic components?
Inspired by previous work of \cite{GradientBasedLearning} we ask whether one can
apply the categorical construction of $\Para$ and $\Lens$ on a particular base
category to model gradient based learning in this setting. That is, we answer
the following four questions:

\begin{itemize}
\item What base category correctly models the adjacency-matrix sharing aspect of
  GCNNs?
\item What is the relation of that base category to the $\Para$ construction?
\item Can one meaningfully backpropagate in that base category, i.e. is that
  category a reverse derivative category?
\item Finally, are the above answers compositional in the manner described in
  \cite{GradientBasedLearning}[Section 3.1]?
\end{itemize}

In addition to answers to the above, the main contribution of the paper is the
proof that there exists a faithful 2-functor $\GCNN_n \to \Para(\CoKl(\adjobj \times
-))$. This functor tells us that GCNNs are a special kind of parametric cokleisli
morphisms, suggesting further generalisations.

\section{Graph Convolutional Neural Networks as Parametric Lenses on base $\CoKl(A \times -)$}

In this section we recall the well-understood categorical construction ${\CoKl(A
\times -)}$, the CoKleisli category of the $A \times -$ comonad for some $A : \cC$, and describe its role in Graph Convolutional Neural Networks.
We do that by proving two important things:

\begin{itemize}
\item $\CoKl(A \times -)$ is a $\cC$ actegory. This permits us to apply the
  $\Para$ construction to it, and show that GCNNs are morphisms in
  $\Para(\CoKl(A \times -))$, for a particular choice of $A$ and base category $\cC$.
\item $\CoKl(A \times -)$ is a reverse derivative category when $\cC$ is. This
  permits us apply the reverse-derivative category fromework to it, and
  compositionally perform backpropagation.
\end{itemize}

This will allow us to simply instantiate the categorical framework presented in
\cite{GradientBasedLearning} and show that GCNNs are recovered in analogous way,
but with a different base category.
We proceed by first recalling the product comonad, and its CoKleisli
category.

\begin{definition}
  Let $\cC$ be a cartesian category.
  Fix an object $A : \cC$.
  Then we can define the \emph{product comonad} $A \times - : \cC \to \cC$ with comultiplication $\delta_X \coloneqq \Delta_A \times X$ and counit $\epsilon_X \coloneqq \pi_X$, where $\Delta_X : A \to A \times A$ is the copy map, and $\pi_X : A \times X \to X$ is the projection.
\end{definition}

Each comonad has an associated CoKleisli category.
We proceed to unpack the details of the CoKleisli category of the above comonad, which we write as $\CoKl(A \times -)$.
It's objects are objects of $\cC$, and a morphism $X \to Y$ in $\CoKl(A \times -)$ is a morphism $A \times X \to Y$ in the underlying category $\cC$.
Composition of $f : \CoKl(A \times -)(X, Y)$ and $g : \CoKl(A \times -)(Y, Z)$ is defined using $\delta_X$; the resulting morphism in $\CoKl(A \times -)(X, Z)$ is
the following composite in $\cC$:
\[
A \times X \xrightarrow{\delta_X} A \times A \times X \xrightarrow{A \times f} A
\times Y \xrightarrow{g} Z
\]

Just like $\Para(\cC)$, the category $\CoKl(A \times -)$ has a convenient
graphical language for describing morphisms (Figure \ref{fig:cokl_composition}).
A composition of three morphisms $f, g, h$ in $\CoKl(A \times -)$ can be
depicted as a string diagram where the global context $A$ is shared on the
vertical side. Every time we want to compose another morphism, we need to plug
in the same global state $A$ used for other morphisms - this is done by copying
on the vertical direction.

\begin{figure}[H]
  \centering
  \includegraphics[width=0.8\textwidth]{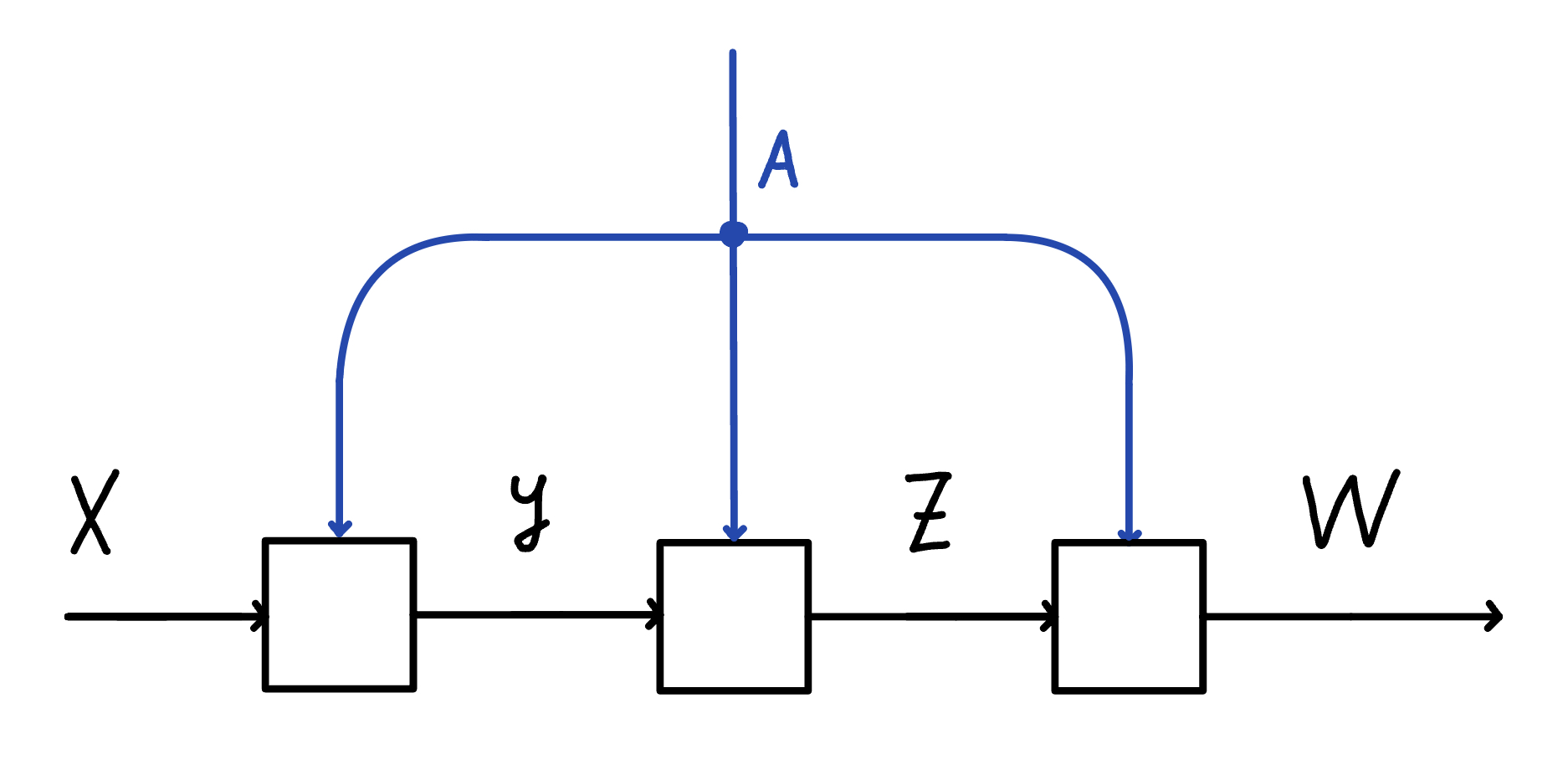}
  \caption{Unlike $\Para$ composition which allows each morphism to be
    parameterised by an arbitrary object, the $\CoKl(A \times -)$ construction
    fixes a specific object $A$ to be the parameter for all of them.}
  \label{fig:cokl_composition}
\end{figure}

A useful lemma for the remainder of the paper will be to show that the category
$\CoKl(A \times -)$ can be given a monoidal structure with the cartesian product.

\begin{lemma}
Let $\cC$ be a cartesian category. Then the category $\CoKl(A \times -)$ is
cartesian monoidal, where the product of morphisms $f : \CoKl(A \times -)(X, Y)$
and $g : \CoKl(A \times -)(X', Y')$ is the morphism of type $\CoKl(A
\times -)(X \times X', Y \times Y')$ defined as

\[
A \times X \times X' \xrightarrow{\Delta_A \times 1_{X \times X'}} A \times A
\times X \times X' \xrightarrow{\textrm{swap}} A \times X \times A \times X'
\xrightarrow{f \times g} Y \times Y'
\]
\end{lemma}

In the next two subsections we proceed to show two things: 1) that $\CoKl(A \times -)$ is a $\cC$-actegory, and 2) that $\CoKl(A \times -)$ is a reverse derivative category when its underlying base is.

\subsection{$\CoKl(A \times -)$ is a $\cC$-actegory}

We've shown how to add a global context $A$ to morphisms in a category. In the
context of graph neural networks this will allow the each morphism to use the
data of a particular adjacency matrix $A$ to broadcast information from each
node to its neighbours.
However, what we'll be interested in is \emph{learning how to} broadcast this
information, hence this is why we will need the $\Para$ construction.

As $\Para$ is a construction defined for any monoidal category, we might simply
be satisfied by using the monoidal structure of $\CoKl(A \times -)$.
However, the situation is slightly more nuanced. In the context of graph neural
networks we do not expect arbitrary reparameterisations to have access to the
global context $A$, hence a simple monoidal category will not suffice.
Luckily, the $\Para$ construction in broader generality can be applied to any
\emph{actegory} \cite{Actegories}, a particular generalisation of a monoidal
category.
Simply speaking, given a monoidal category $\cM$, an $\cM$-actegory $\cC$ is given by
a strong monoidal functor $\cM \to \internal{\Cat}{\cC}{\cC}$. We can easily see
that any monoidal category $\cM$ is an actegory acting on itself, and is given by currying of the monoidal product $\otimes$.

In our case, we'll be interested in in a $\cC$-actegory $\CoKl(A \times -)$. In
order to define it, we'll make use of the fact that the embedding $\cC \to
\CoKl(A \times -)$ preserves finite products (i.e. is strong monoidal).

\begin{lemma}\label{lemma:base_into_cokleisli}
Given a monoidal category $\cC$, there is a strict symmetric monoidal
identity-on-objects functor $\iota : \cC \to \CoKl(A \times -)$ which maps a morphism $f
: X \to Y$ to $A \times X \xrightarrow{\pi_1} X \xrightarrow{f} Y$.
\end{lemma}

\begin{proof}
The monoidal structure is trivially preserved on objects. On morphisms we have to
check whether there is an isomorphism $\iota(f \times g) \cong \iota(f) \times
\iota(g)$, where $f : X \to Y$ and $g : X' \to Y'$.
On the left side, this unpacks to a morphism $$A \times X \times X'
\xrightarrow{\pi_{X \times X'}} X \times X' \xrightarrow{f \times g} Y
\times Y'$$
On the right side, we have that $\iota(f) : A \times X \xrightarrow{\pi_X} Y$ and
$\iota(g) : A \times X' \xrightarrow{\pi_Y} Y'$. Their product in $\CoKl(A
\times -)$ is the morphism

\[
  A \times X \times X' \xrightarrow{\Delta_A \times 1_{X \times X'}} A \times A
  \times X \times X' \xrightarrow{\textrm{swap}} A \times X \times A \times X'
  \xrightarrow{(\pi_X \comp f) \times (\pi_X' \comp g)} Y \times Y'
\]
which is equivalent to the one above, concluding the proof.
\end{proof}

The existence of the strict monoidal functor $\iota$ allows us to get the
desired actegory. That is, instead of thinking of $\CoKl(A \times -)$ as acting
on itself, we can think of $\cC$ as acting on $\CoKl(A \times -)$.

\begin{definition}
  We define the \emph{non-contextual action} $\odot$ of $\cC$ on $\CoKl(A \times
  -)$ as
  \[
    \cC \xrightarrow{\iota} \CoKl(A \times -) \xrightarrow{\textrm{curry } \times} \internal{\Cat}{\CoKl(A \times
      -)}{\CoKl(A \times -)}.
  \]
\end{definition}

This allows us to finally apply the $\Para$ construction to the non-contextual
coeffect, yielding the bicategory relevant for the rest of the paper, which we
refer to as the \emph{Parametric CoKleisli category.}

Unpacking $\Para_{\odot}(\CoKl(A \times -))$, we get a bicategory whose objects
are objects of $\cC$; and a morphism $X \to Y$ is a choice of a parameter object
$P : \cC$ and a morphism $f : P \times A \times X \to Y$.
A 2-cell $(P', f') \Rightarrow (P, f)$ is a choice of a morphism $r : P \to P'$
such that the diagram below commutes. 
\[\begin{tikzcd}
    {P' \times A \times X} && {P \times A \times X} \\
    & Y
    \arrow["{r \times A \times X}", from=1-1, to=1-3]
    \arrow["f", from=1-3, to=2-2]
    \arrow["{f'}"', from=1-1, to=2-2]
\end{tikzcd}\]

Two morphisms $(P, f) : X \to Y$ and $(Q, g) : Y \to Z$ are composed by copying
the global state $A$, and taking the product of the parameters $P$ and $Q$.
That is, their composite is the morphism $(Q \times P, h) : X \to Z$, where $h$
unpacks to
\begin{align}\label{eq:para_cokl_composition}
  \begin{split}
Q \times P \times A \times X \xrightarrow{1_{Q \times P} \times \Delta_A \times
  X} Q \times P \times A \times A \times X \xrightarrow{\textrm{swap}} \\
\to Q \times A \times P \times A \times X \xrightarrow{1_{Q \times A} \times f} Q \times A
\times Y \xrightarrow{g} Z
  \end{split}
\end{align}

This defines the data of a composite morphism drawn in Figure
\ref{fig:graph_convolutional}. We see that the global data $A$ is threaded through
on one axis, while the parameters are used on another. This concludes the definition of our base category $\Para_{\odot}(\CoKl(A \times
-))$.

We are now one step away from
seeing how graph neural networks arise in a compositional way.
The last missing piece is the deferred definition of 2-cells of $\GCNN_n$, which
we present below.

\begin{definition}[2-cells in $\GCNN_n$]\label{def:two_cells}
  Given two graph convolutional neural networks
  \[
    h : {\displaystyle \prod_{i : \overline{1 + m + 1}} \R^{d_i \times d_{i +
          1}}} \times \R^{n \times n} \times \R^{n \times k} \to \R^{n \times l}
  \]
  and
  \[
    h' : {\displaystyle \prod_{i : \overline{1 + m' + 1}} \R^{d'_i \times d'_{i
          + 1}}} \times \R^{n \times n} \times \R^{n \times k} \to \R^{n \times
      l}
  \]
  (where $m$ and $m'$ are numbers of layers of each network, $d_i$ and $d'_i$
  are dimensions of these layers, such that $d(0) = k$ (the dimensionality of
  the incoming layer) and $d(m + 1) = k'$ (the dimensionality of the incoming
  layer)) a 2-cell between them is a smooth function $r : {\displaystyle
    \prod_{i : \overline{1 + m + 1}} \R^{d_i \times d_{i + 1}}} \to
  {\displaystyle \prod_{i : \overline{1 + m' + 1}} \R^{d'_i \times d'_{i + 1}}}$
  such that $(r \times \R^{n \times n} \times \R^{n \times k}) \comp h = h'$.
\end{definition}

This enables us to state our main theorem in this paper.
We first fix $\Smooth$ to be the cartesian category of euclidean spaces and smooth functions between them \cite[Example 2.3]{ReverseDerivativeCategories}.
We then choose an object to define the product comonad on -- the type of the adjacency matrices of a graph with $n$ nodes -- $\R^{n \times n} : \Smooth$.
Now are are able to show that \textbf{Graph Convolution Neural Networks are
  morphisms in ${\pc}$}. We do that by proving there is an injective-on-objects
and faithful 2-functor embedding them into the aforementioned category.

\begin{theorem}\label{thm:two_functor}
  There exists an injective-on-objects, faithful 2-functor
  \[
    \kappa : \GCNN_n \to \pc
  \]
where $\adjobj \times -$ is the product comonad defined on $\Smooth$.
\end{theorem}

\begin{proof}
  $\kappa$ is identity on objects.
  On morphisms, it maps a GCNN
  \[
    h : {\displaystyle \prod_{i : \overline{1 + m + 1}} \R^{d_i \times d_{i + 1}}} \times \R^{n \times n} \times \R^{n \times k} \to \R^{n \times l}
  \]
  to a pair $({\displaystyle \prod_{i : \overline{1 + m + 1}} \R^{d_i \times
      d_{i + 1}}}, h)$, turning the parameter spaces explicit.
  Here ${\displaystyle \prod_{i : \overline{1 + m + 1}} \R^{d_i \times d_{i +
        1}}}$ is the local parameter space, and $R^{n \times n}$ is the global one.
  Following Eq. \ref{eq:para_cokl_composition}, we see that a composition of
  morphisms in this category correctly replicates the composition rule defined
  in Def. \ref{def:prim_composition}.
  Likewise, it's easy to see that 2-cells induced by $\Para$ are exactly the
  those of $\GCNN_n$.
  The coherence conditions are routine.
\end{proof}

This tells us that we can really focus on $\pc$ as our base category, and work
with arbitrary $n$-node GCNNs that way.

\subsubsection{Aside: Interaction between $\CoKl(A \times -)$ and $\Para(\cC)$}

The interested reader might have noticed a similarity between the $\CoKl(A
\times -)$ and
$\Para(\cC)$ constructions. They both involve a notion of a parameter, but they
differ in how they handle composition and identities. The former one is always
parameterised by the same parameter, while the latter one allows each morphism
to choose their parameter.
On identity maps, the former one picks out a parameter that will be deleted,
while the latter one picks out the monoidal unit.
When composing $n$ morphisms, the former takes in \emph{one} input parameter and copies it $n$ times -- one time for each morphism -- while the latter one takes in $n$ different parameter values of possibly different types, and relays each to the corresponding constituent morphism.
We refer the reader to appendix B where these similarities are explored in terms
of an oplax functor between these constructions, and to the table below
for a high-level outline of differences.

\begin{table}[h]
  \begin{center}
    \caption{Both $\Para$ and $\CoKl(A \times -)$ are oplax colimits of
      2-functors, but differ in key aspects.}
    \begin{tabular}{|l|l|l|}
      \hline
      & $\Para$        & $\CoKl(A \times -)$ \\ \hline
      Parameter context    & Local          & Global              \\ \hline
      Type of construction & Bicategory     & Category            \\ \hline
      Required structure   & Graded Comonad & Comonad             \\ \hline
    \end{tabular}
  \end{center}
\end{table}

\subsection{$\CoKl(A \times -)$ is a reverse derivative category when $\cC$ is}\label{subsec:cokl_rdc}

Following in the footsteps of \cite{GradientBasedLearning}, in addition to requiring that $\CoKl(A \times -)$ is an actegory, we need to show we can in a sensible way \emph{backpropagate} through this category.
Formally, this means that $\CoKl(A \times -)$ is a reverse-derivative category, first defined in \cite{ReverseDerivativeCategories}.

An intuitive way to think about CoKleisli categories as reverse derivative categories is using partial derivatives.
As described in \cite[page 6.]{ReverseDerivativeCategories}, the forward derivative of $f : \CoKl(A \times -)(X, Y)$ is interpreted as a partial derivative of the underlying $f : A \times X \to Y$ with respect to $B$.
Its \emph{reverse} derivative is interpreted as the \emph{transpose} of same partial derivative.
This brings us to the second theorem of this paper.

\begin{restatable}[CoKleisli of a RDC is RDC]{theorem}{CoKleisliIsRDC}\label{prop:cokl_is_rdc}
  Let $\cC$ be a reverse derivative category. Fix an object $A : \cC$. Then
  $\CoKl(A \times -)$ is too a reverse derivative category whose reverse
  differential combinator for a map $f : \CoKl(A \times -)(X, Y)$ is defined as

  \[
    R_{\CoKl(A \times -)}[f] \coloneqq R_{\cC}[f] \comp \pi_X
  \]
  where $f, R_{\cC}[f]$, and $\pi_X$ are all treated as morphisms in $\cC$.
\end{restatable}

\begin{proof}
  Appendix.
\end{proof}

In the framework of \cite{GradientBasedLearning}, a category $\cC$ being a
reverse-derivative is interpreted as a product-preserving functor $\cC \to
\Lens_A(\cC)$, where $\Lens_A(\cC)$ is the category of bimorphic lenses where
the backwards map is additive in the 2nd component. We refer the reader to
\cite{GradientBasedLearning} for more details.

Then, theorem \ref{prop:cokl_is_rdc} can be interpreted as a lifting: any
functor which takes a a category $\cC$ and augments it with the backward pass

\[
  R : \cC \to \Lens_A(\cC)
\]

can be lifted to an functor which takes the CoKleisli category of the
product comonad defined on $\cC$ and augments \emph{that} category with its backward
pass:

\[
  R_{\CoKl(A \times -)} : \CoKl(A \times -) \to \Lens_A(\CoKl(A \times -))
\]

But that is not all, as the category of interest for us isn't $\CoKl(A \times
-)$, but rather $\Para(\CoKl(A \times -))$.
This brings us to the third theorem of this paper, which shows that \textbf{we
  can compositionally backpropagate through Graph Convolutional Neural
  Networks}.

\begin{theorem}\label{thm:para(r_cokl)}
  There is an injective-on-objects 2-functor
  \begin{equation}
    \Para(R_{\CoKl(A \times -)}) : \Para(\CoKl(A \times -)) \to \Para(\Lens_A(\CoKl(A \times -)))
  \end{equation}
  which augments a parametric cokleisli morphism with its reverse derivative.
\end{theorem}

\begin{proof}
Follows from applying Remark 2.1 in \cite{GradientBasedLearning} to $R_{\CoKl(A
  \times -)}$.
\end{proof}

We proceed to unpack the action of this relatively involved 2-functor, and show it
correctly models differentiation in this setting of GCNNs.

On objects, this functor maps $X$ to $(X, X)$. A morphism $(P, f : P \times A
\times X \to Y)$ -- intepreted as a locally $P$-parameterised and globally
$A$-parameterised map -- gets mapped to a pair whose first element is $(P, P)$
and the second element is a morphism $(P \times X, P \times X) \to (Y, Y)$ in
$\Lens_A(\CoKl(A \times -))$.
This means that this is a lens whose forward and backward part are both locally
$P$-parameterised, and globally $A$-parameterised.

This lens is defined by morphisms $f :
\CoKl(A \times -)(P \times X, Y)$ and $(R[f] \comp \pi_{P \times X} : \CoKl(A \times -)(P \times X \times Y, P \times X))$.
We note that the local parameter $P$ has appeared in the codomain of the backward map, which is not the case for the global parameter $A$.
This tells us that we are interested in computing the gradient with respect to $P$, but are not interested in computing the gradient with respect to $A$.
As mentioned in the beginning of subsection \ref{subsec:cokl_rdc}, we can think of this as taking the partial derivative of $f$ with respect to $P$ and $X$, but not $A$.

Unpacking the types completely, we see that the forward map is the map $f : A \times P \times X \to Y$ in $\cC$, and that the backward map is $R[f] \comp \pi_{P \times X} : A \times P \times X \times Y \to P \times X$ in $\cC$.
This means that our forward map $f$ was unchanged, and that the backward map computes the reverse derivative of $f$, and then only uses the computed derivative for $P$ and $X$.

This now allows us to finally show how a concrete GCNN can be backpropagated
through -- and this is by composing the functors in Theorems \ref{thm:two_functor} and
\ref{thm:para(r_cokl)}:

\[\begin{tikzcd}
    {\GCNN_n} \\
    \pc \\
    \plc
    \arrow["{\kappa_n}", from=1-1, to=2-1]
    \arrow["{\Para(R_{\CoKl(\R^{n \times n} \times - )})}", from=2-1, to=3-1]
  \end{tikzcd}\]

Unpacking the construction fully, we see that each object $\R^{n \times k}$ in
$\GCNN_n$ turns into a pair thereof: one describing values being computed,
the other one derivatives.
Each Graph Convolutional Neural Network turns into a lens which on the forward
pass computes the forward part of this GCNN - correctly broadcasting node
information to all other nodes.
The backward part of this lens then performs gradient descent on that forward
part, accumulating gradients in all the relevant parameters.

This enables us to simply plug in the rest of the framework described in
\cite{GradientBasedLearning} and add a loss function, optimisers, and finally
\emph{train} this network.

This concludes our formalization of the Graph Convolutional Neural Network
in Category theory through the CoKleisli construction. What follows are sketches
towards compelling avenues of research in generalising this construction
to comprehend the theory of neural architecture in greater detail and
to find disciplined approaches for the selection of an architecture given
specific priors.

\subsubsection{Aside: Towards the Equivalence for $\textbf{GCNN}_n$ and the
  parametrised lenses construction}

The title of this paper states that all graph convolutional neural networks are
parametric cokleisli morphisms, but interestingly \emph{not all} parametric cokleisli morphisms are graph convolutional neural networks.
Mathematically, this can be stated as the failure of $\kappa_n$ to be a full
2-functor. The 2-functor $\kappa_n$ is faithful and identity-on-objects, but it does not give us an equivalence of categories $\GCNN_n$ and $\plc$.

This can be tracked down to the base category $\Smooth$, and is an issue appearing in the framework of \cite{GradientBasedLearning}. The kinds of morphisms we consider to be neural networks are usually not just arbitrary smooth maps, but rather those formed by sequences of linear functions and activation functions (i.e. non-linear functions of a particular form).

On the other hand, using $\Vect$ as a base category would limit us to parametrised linear functions as morphisms. As we wish our construction to be faithful to the architectures deployed in the AI literature \cite{kipf2016semi} this inability to explicitly account for the nonlinear activation function presents an obstacle to this.
Indeed, the application of non-linear activations functions does not commute with composition of linear functions $\sigma(X \cdot X') \neq \sigma(X) \cdot\sigma(X')$.

This is truly inherent to universal approximation, which require the activations functions to be generally non-polynomial in order for neural networks to be universal approximators \cite{cybenko1989approximation}.
However, piece-wise linear functions fit this definition and are broadly used in deep learning. Consider for instance the rectifier or Rectified Linear Unit (ReLU) activation, which is made of two linear components and is applied componentwise to a vector of preactivation $x \in \mathbb{R}^k$.

\begin{definition}
  The Rectified Linear Unit is defined as
  \begin{align*}
    \ReLU: \R \rightarrow \R \coloneqq x \mapsto \max(x, 0)
  \end{align*}
We often use the same notation $\ReLU$ for the pointwise application of this map
$\ReLU^{k} : \R^k \to \R^k$ to each component $x_i$ of $x : \R^k$
\end{definition}

Much of what follows can be extended to general piecewise linear functions, with the appropriate considerations, but in a first instance we only consider ReLU as a topical example, as well as a broadly used activation. When we examine the ReLU function, in particular, we find that it satisfies the following property:

\begin{lemma}
  For every $x \in \mathbb{R}^k$, there exists a vector $p \in \{0,1\}^k$ such
  that
  $$\diag(p) \cdot x = \ReLU(x)$$
  where $\diag(p) : \R^{k \times k}$ is a diagonal matrix with values of $p$ on the diagonal.
\end{lemma}

It is easy to see why this is true: the vector $p$ is describing which component of $x$ is greater than or equal to $0$. Applying $\diag(p) \cdot - $ to $x$ then provides the equivalent of killing the components that would be sent to zero in $\ReLU$.

\[\begin{tikzcd}
  {\mathbb{R}^n} \\
  \\
  \\
  {\mathbb{R}^m} \\
  \\
  {\mathbb{R}^m}
  \arrow["{-^T \cdot W}"{description}, color={rgb,255:red,204;green,51;blue,51}, from=1-1, to=4-1]
  \arrow["{\text{ReLU}(-)}"{description}, from=4-1, to=6-1]
  \arrow["{NN_{layer}(-)}"', curve={height=30pt}, from=1-1, to=6-1]
  \arrow["{diag(P)\cdot -}"{description}, shift left=3, color={rgb,255:red,204;green,51;blue,51}, curve={height=-30pt}, from=4-1, to=6-1]
  \arrow["{NN_{layer}(-)}"{description}, shift left=4, color={rgb,255:red,205;green,55;blue,55}, curve={height=-30pt}, from=1-1, to=6-1]
\end{tikzcd}\]
In the diagram above we show how every ReLU function in $\textbf{Smooth}$ can be represented locally as a linear function. Indeed, the red morphisms are linear functions and this is the key result of this observation: for ReLU activations there always exists a linear function that replicates exactly the local behaviour of $ReLU$ at $x$.

In summary, we postulate that the appropriate setting to do deep learning is a category more expressive than $\Vect$, but not quite as rich as $\textbf{Smooth}$.

	\section{Directions Beyond Graph Convolutions}
In \cite{bronstein2021geometric}, the authors suggest that Category Theory may play a unifying role in Geometric Deep Learning (GDL). Indeed there are two ways in which Category Theory may be able to assist in such efforts of unveiling a theory of architecture for neural networks: first, many artefacts in the engineering of deep learning systems are applied on an experimental basis. Here Category theory may provide a principled approach to selecting layers or hyperparameters (such as dropout, batch normalization, skip connections etc.) that may discipline the current practices in deep learning, in a similar vein to how functional programming has disciplined the design space of programming languages.

Secondly, the exploration of neural architectures resulting from inductive biases, such as observations of the data domain (the datapoint lives on a node of a graph, the network should be equivariant to the symmetries of a Lie group etc.) often inhibit the generalisation of neural architectures to general constructions. The equivariance literature looks at how to design networks that are indifferent to certain group transformation \cite{keriven2019universal}; meanwhile, networks that are increasingly expressive \cite{bodnar2021weisfeiler} or that are tailored around specific geometric meshes \cite{de2020gauge}. While these accomplishments are remarkable, a desirable next step is to abstract the theory of architecture in a domain-agnostic way. Such a theory can specify a network given an inductive bias, minimising the number of decision an AI-engineer must make, while streamlining the efforts of deep learning research.

A contribution by \cite{bronstein2021geometric} was to offer a first step in this direction, describing a general GDL architecture as a composition of layers in specific order:

\[
f = O \circ E_l \circ \sigma_l \circ P_l \circ ... \circ E_1 \circ \sigma_1 \circ P_1
\]

where $\{P_i\}_{i\in [l]}$ represent pooling layers, $\{E_i\}_{i \in [l]}$ are locally equivariant layers and $\{\sigma_i\}_{i\in [l]}$ are activation functions. This framework is general enough to encompass the family of aforementioned architectures. Nevertheless, there are artefacts that are not clearly detailed in this framework (skip connections, dropout etc.) and attempting to include them would be very involved. We conjecture that the strain would be inferior were we to represent these models categorically.

We have seen in the earlier sections how GCNNs have two types of parametrization, a \textit{global} adjacency matrix and a weight \textit{matrix}. Can we reconcile these two decomposition of geometric-based neural networks?

Many architectural choices are due to the structure of the input. In fact, as an inductive bias, the composition of layers is a modelling decision that is derived from the observation of certain properties in the data. The most immediate patterns arise in the genre of data: time series, images, text or numerical tables - each of these prompting an explicit design choice - however, there can be more subtle patterns that emerge in how we encode the features (categorical, ordinal encoding etc.).

In particular, viewing the CoKleisli category construction as a \emph{top-down} approach, in the sense that we dissect the architecture of a GCNN to it's main components, we now take the \emph{bottom-up} view of constructing an architecture based on the known restrictions of the problem domain. Geometric Deep Learning (GDL) \cite{bronstein2021geometric} is the discipline that explores the relationship between problem symmetries and architecture. For example, in a picture we may request that the network be invariant to translational and rotational symmetries: this fact can be encoded through the structure of a group, discretised on a grid, that allows us to ensure that all inputs in a given mode of a tensor are treated similarly. This is the case for Convolutional Neural Networks (more on the subject in \cite{gu2018recent}), which may be used to find classifiers for objects in a picture (such as the MNIST digit classification).

Overall, the role of batch processing has been hinted to in the previous sections: let us elucidate how these design choices affect architecture.
A way to represent neural networks on a Graph was to consider the functor:

\[
F: \mathbb{B}\text{Aut}(G) \rightarrow \textbf{Smooth},
\]

the classifying space of the automorphism group of the graph $G$, viewed as a category, into the category of euclidean spaces and smooth functions between them.\footnote{We thank Matteo Capucci for this observation.}
This functor maps every group object to the space $\mathbb{R}^n$, where $n = |V|$ the size of the vertex set of the graph; every group transformation, viewed as an automorphism in the object, is then realised as a permutation of $k$ bases.
This is to be understood as the family of functions which are equivariant with respect to $\text{Aut}(G)$.
In essence, we encode a desiderata into the functional class. Can we infer, from this prior, that the architecture we want will be built over $\CoKl(A \times -)$?

To sketch our initial considerations in addressing this problem, we present a sequence of functors that allow us to assign an architecture to the graph $G: \textbf{Span}(\textbf{FinSet})$. Starting with a pseudofunctor from $\textbf{Span}(\textbf{FinSet}) \rightarrow \textbf{Cat}$ that assigns to every graph its automorphism group as the category $\mathbb{B}\text{Aut}(G)$. A functor $R : \mathbb{B}\text{Aut}(G)\rightarrow \Vect$ is a representation over the group of automorphisms of $G$; the category of linear representations $[\mathbb{B}\text{Aut}(G), \Vect]$ is then enclosing all such assignments. A morphism in this category is a linear transformation $A: \mathbb{R}^k \rightarrow \mathbb{R}^k$ on the only image of $\mathbb{B}\text{Aut}(G)$: yielding a $k\times k $-matrix. Note that said collection of matrices need to compose like the morphisms in $\mathbb{B}\text{Aut}(G)$, which leaves us few options.

We can also speculate on how to generalise this procedure: we have generally assumed that the input to a vanilla neural network could be generally represented through a vector in $\R^n$. Ultimately that's correct, but only after some of the categorical inputs have been transformed to vectors, through some encoding procedure such as one-hot encoding or ordinal encoding, depending on the properties of the data. Notwithstanding the true nature of the data, it is important that the inputs are then transformed to vectors so that backpropagation may apply to the network.

However, this observation can allow us to abstract the learning problem, from a rich structure built over $\Smooth$, to something living in a higher category, perhaps $\textbf{Cat}$ itself. For instance, a collection of features may be a one hot encoding of a set, or an ordered set $X$, viewed as a category. Consider the functor into $\Smooth$, that applies the encoding to every object and morphism of $X$.

The above observations apply for a single feature, such as the city or borough where a house is (taken from a set of strings), or the floor where the apartment is located in a building (an ordered set of integers). Let us take a simple example in which data points are composed of a product $G \times \mathbb{R}^n$, the latter viewed as a set. $G$ in this case is a categorical feature. This could be a string, or an ordered set, or perhaps the names of a nodes on a graph. We may decide to encode this in the network, through a mapping $G\rightarrow \mathbb{R}^l$ for the appropriate $l\in \mathbb{Z}$, or we could build a network on $[\mathbb{B}\text{Aut}(G),\textbf{Smooth}]$, wherever this can be built. In the example of graph neural networks, we may either decide to encode the position on a graph on a categorical variable that is given as a feature to the network or build a neural network on each of the vertices of the graph, as GCNs can be understood.

As a final note, we want to suggest that we should define a category where the input data points live, which we dub a \textbf{batch category}. From there, the trick to completing the relationship between the CoKleisli construction and the functor category as characterisation of architecture may become clearer. In summary, here we outlined a research agenda aiming at the specification of the inductive bias through the properties of the input space. This in turn helps us navigate architecture space with a stronger conviction that the architectures are better suited for the task at hand, an overarching goal of Geometric Deep Learning.





	\section{Conclusions}

We have defined the bicategory of Graph Convolutional Neural Networks, and shown
how it can be factored through the already existing categorical constructions
for deep leanring: $\Para$ and $\Lens$.
We've shown that there is an injective-on-objects, faithful functor $\GCNN_n \to
\pc$, and through this construction gained an insight on the local/global
aspect of parameters in GCNNs.
In addition to a local parameter for each layer in a GCNN, we've shown a unique
characteristic of GCNNs as neural networks whose each layer has access to a
global parameter in addition to a local one.
This describes part of the inductive bias of GCNNs in a purely categorical
language, in terms of the CoKleisli category of the product comonad.

There is much more to do. We outline some key areas for development:
\begin{itemize}
\item GCNNs are neural networks with a very particular form of message passing.
  The message passing aspect described in \cite[Eq. (1)]{GNNDynProg} suggests a
  much more sophisticated form, as it closely matches a morphism in the category
  of dependent lenses. Are general graph neural networks morphisms in $\mathbf{DLens}(\Lens_A(\cC))$?
\item The framework of $\Para(\Lens(\cC))$ is closely related to the framework
  used for game-theoretic modelling of economic agents. Does studying games on
  graphs \cite{NetworkGames} tell us anything about graph neural networks, or
  vice versa?
\item Can the categorical construction in this paper tell us anything about
  updating the adjacency matrix too? As $\CoKl(A \times -)$ is too a
  Grothendieck construction of a particular graded monad, can we apply the same
  tool to lift a reverse derivative functor, but this time in a way which
  computes the gradient too?
\item Is there a particular base category which we can instantiate the
  $\Para(\Lens)$ framework on such that (graph) neural networks are not merely a
  subcategory of, but equivalent to?
\end{itemize}


%
%
%
%
%

	\bibliographystyle{alpha}
	\bibliography{bibliography}

  \appendix
  \section{Appendix}

\CoKleisliIsRDC*
\begin{proof}
  Recall that a reverse derivative category is equivalent to a forward derivative category with a contextual linear dagger (\cite[Theorem 42.]{ReverseDerivativeCategories}).
  This allows us to restate both the starting data and the end goal of the proof using forward derivative categories, for which we can use existing machinery.
  To prove $\CoKl(A \times -)$ is a forward derivative category given that the base $\cC$ is forward derivative, we use \cite[Proposition 8.]{ReverseDerivativeCategories}\footnote{The cited resource calls $\CoKl(A \times -)$ a \emph{simple slice category}, inspired by the idea that there's a faithful embedding $\CoKl(A \times -) \hookrightarrow \cC/A$.}.
  What remains is to prove that $\CoKl(A \times -)$ has a contextual linear dagger (on the forward derivative structure given above) given that the base does (on the forward derivative structure assumed above). The ``contextual'' here refers to the fact that the existing linear fibration has a dagger structure. This follows in a straightforward way since the linear fibration associated to a cartesian differential category does in fact have a dagger structure.
\end{proof}

\section{Categorical interaction between $\CoKl(A \times -)$ and $\Para$}

The fact that CoKleisli categories of the $A \times -$ comonad describe a kind
of parameterised processes where the parameter is always equivalent to $A$ can
make us ponder the question: is there a functor $\CoKl(A \times -) \to
\Para(\cC)$? The answer is yes.

\begin{proposition}[$\CoKl(A \times -)$ embeds into $\Para(\cC)$]
  Let $\cC$ be a cartesian category. Let $A : \cC$. Then there is an oplax,
  identity-on-objects functor $\tau : \CoKl(A \times -) \to \Para(\cC)$ which takes a
  morphism $f : \CoKl(A \times -)(X, Y)$ and maps it to the $A$-parameterised
  morphism $(A, f) : \Para(\cC)(X, Y)$.
\end{proposition}

\begin{proof}
  We unpack the definition of the unit and composition 2-cells, and omit the
  proof of their coherence, which is routine.
Given two morphisms $f : \CoKl(A \times -)(X, Y)$,  and $g : \CoKl(A \times
-)(Y, Z)$, there are two ways to end up with a morphism in $\Para(\cC)(X, Z)$.

\[\begin{tikzcd}[column sep=2.5ex]
	& {} \\
	{\CoKl(A \times -)(X, Y) \times \CoKl(A \times -)(Y, Z)} && {\Para(\cC-)(X, Y) \times \Para(\cC)(Y, Z)} \\
	\\
	{\CoKl(A \times -)(X, Z)} & {} & {\Para(\cC)(X, Z)}
	\arrow["\comp"', from=2-1, to=4-1]
	\arrow["\tau"', from=4-1, to=4-3]
	\arrow["{\tau \times \tau}", from=2-1, to=2-3]
	\arrow["\comp", from=2-3, to=4-3]
	\arrow[shorten <=29pt, shorten >=44pt, Rightarrow, from=2-3, to=4-1]
	\arrow["{\delta_{X, Z}}"{description, pos=0.4}, draw=none, from=4-2, to=1-2]
\end{tikzcd}\]

First applying $\tau$ to both $f$ and $g$ and then composing them as parametric
morphisms yields a morphism $(A \times A, (A \times \tau(f)) \comp \tau(g))$.
On the other hand, if we first compose them, and then apply $\tau$ we obtain a
morphism $(A, (\Delta_A \times X \comp f) \comp g)$.
These are connected by a reparameterisation $\tau(f) \comp \tau(g) \Rightarrow
\tau(f \comp g)$, i.e. a copy morphism $\Delta_A : A \to A \times A$.

Likewise, this functor preserves identities also only up to a 2-cell. Starting
with an object $X : \CoKl(A \times -)$, there are two ways to obtain a morphism
in $\Para(\cC)(X, X)$.
We can either take the identity in $\pi_X : \CoKl(A \times -)$ and then apply
$\tau$ to it, or we can look at $\tau(X)$ (which is equal to $X$) and look at
its identity in $\Para(\cC)$, which unpacks to $(1, \lambda_X)$.
These morphisms are too connected by the terminal reparameterisation $! : A \to 1$.
\end{proof}

\end{document}